\documentclass[12pt]{amsart}
\usepackage[margin=1.2in]{geometry}
\usepackage{graphicx,latexsym}
\usepackage{amsfonts, amssymb, amsmath, amsthm}

\newtheorem{theorem}{Theorem}
\newtheorem{proposition}{Proposition}
\newtheorem{lemma}{Lemma}
\newtheorem{corollary}{Corollary}

\theoremstyle{definition}
\newtheorem{definition}{Definition}

\theoremstyle{remark}
\newtheorem{remark}{Remark}

\begin{document}

\title[New distribution spaces]{New distribution spaces associated to translation-invariant Banach spaces}

\author[P. Dimovski]{Pavel Dimovski}
\address{Faculty of Technology and Metallurgy, Ss. Cyril and Methodius University, Ruger Boskovic 16, 1000 Skopje, Macedonia}
\email{dimovski.pavel@gmail.com}

\author[S. Pilipovi\'{c}]{Stevan Pilipovi\'{c}}
\address{Department of Mathematics and Informatics, University of Novi Sad, Trg Dositeja Obradovi\'ca 4, 21000 Novi Sad, Serbia}
\email {stevan.pilipovic@dmi.uns.ac.rs}

\author[J. Vindas]{Jasson Vindas}
\address{Department of Mathematics, Ghent University, Krijgslaan 281 Gebouw S22, 9000 Gent, Belgium}
\email{jvindas@cage.Ugent.be}

\thanks{J. Vindas gratefully acknowledges support by Ghent University, through the BOF-grant number 01N010114.}

\subjclass[2010]{Primary 46F05, 46E10. Secondary 46H25}
\keywords{Schwartz distributions; Translation-invariant Banach spaces of tempered distributions; Convolution of distributions; $L^{p}$ and $\mathcal{D}_{L^{p}}'$ weighted spaces; Beurling algebras}

\begin{abstract}We introduce and study new distribution spaces, the test function space $\mathcal{D}_E$ and its strong dual $\mathcal{D}'_{E'_{\ast}}$. These spaces generalize the Schwartz spaces $\mathcal{D}_{L^{q}}$, $\mathcal{D}'_{L^{p}}$, $\mathcal{B}'$ and their weighted versions. The construction of our new distribution spaces is based on the analysis of a suitable translation-invariant Banach space of distributions $E$, which turns out to be a convolution module over a Beurling algebra $L^{1}_{\omega}$. The Banach space $E'_{\ast}$ stands for $L_{\check{\omega}}^1\ast E'$. We also study convolution and multiplicative products on $\mathcal{D}'_{E'_{\ast}}$.
\end{abstract}

\maketitle

%
%
%

\section{Introduction}

Translation-invariant spaces of functions and distributions are very important in mathematical analysis.
They are connected with many central questions in harmonic analysis \cite{beurling,domar56,domar89,reiter,wermer}. Particularly, the convolution operation makes them relevant for many branches of analysis such as the theory of PDE. The basic properties of the convolution of distributions were already established in the books by  Schwartz \cite{S} and Horv\'{a}th \cite{horvath}, but there is still a wide interest in the subject \cite{BO,O2004,O2010,O1990,O,wagner1987,wagner2013}.

This article introduces and studies new classes of translation-invariant distribution spaces, the test function space $\mathcal{D}_{E}$ and its dual, denoted as $\mathcal{D}'_{E'_{\ast}}$. The construction of these spaces is based upon the analysis of suitable translation-invariant Banach spaces of distributions $E$. 

The space $E$ is a natural extension of a large class of weighted $L^{p}$ spaces, while $\mathcal{D}'_{E'_{\ast}}$ generalizes the spaces $\mathcal D'_{L^p}$. The spaces $\mathcal D'_{L^p}$ were introduced by Schwartz \cite{S54,Schwartzvectorial,S} as a major tool in the study of convolution within distribution theory \cite{BO,O2010} and are still the subject of various modern investigations \cite{OW} (cf. also \cite{O1983}). Ortner and Wagner \cite{O1990} have considered weighted versions of the  $\mathcal D'_{L^p}$ spaces, which have proved usefulness in the analysis of convolution semigroups associated to many PDE \cite{O} and boundary values of harmonic functions \cite{alvarez2007}. It turns out that their spaces are also particular instances of our $\mathcal{D}'_{E'_{\ast}}$.

The present work focuses on the topological and convolution properties of $\mathcal{D}'_{E'_{\ast}}$. The new spaces could be very useful  in the analysis of boundary values of holomorphic functions in tube domains \cite{C-M} and solutions to the heat equation in the upper half-space \cite{M1990}. In fact, the authors have applied the results from this paper to make a thorough analysis of boundary values in the space $\mathcal{D}'_{E'_{\ast}}$, which leads to a series of new results even for the classical case $\mathcal{D}'_{E'_{\ast}}=\mathcal{D}'_{L^{p}}$. Such results will appear elsewhere in a forthcoming article \cite{d-p-vII}.

This paper is organized as follows. In Section \ref{translation-invariant} we study the class of  translation-invariant Banach spaces of tempered distributions on the Euclidean space $\mathbb{R}^{n}$. It should be mentioned that such Banach spaces have already been considered in one-dimension by Drozhzhinov and Zav'yalov in connection with Tauberian theorems and generalized Besov spaces \cite{DZ}. The class of Banach spaces in which we are interested are translation-invariant spaces $\mathcal{D}(\mathbb{R}^{n})\hookrightarrow E \hookrightarrow \mathcal{D}'(\mathbb{R}^{n})$ for which the translation group acts continuously and the growth function $\omega$ of its translation group (cf. Definition \ref{def:1}) is polynomially bounded. See Section \ref{translation-invariant} for precise definitions. We remark that the space $E$ is not necessarily a solid Banach space in the sense of \cite{f1979}; indeed, the elements of $E$ may not be regular distributions and actually $E$ needs not even be a module over $\mathcal{D}(\mathbb{R}^{n})$ under pointwise multiplication. Nevertheless, as we show, the space $E$ carries a natural Banach convolution module structure over the Beurling algebra $L^{1}_{\omega}$. Furthermore, it is shown that $E$ possesses bounded approximations of the unity for this module structure. Actually, this class of Banach spaces turns out to be characterized by the latter two properties (cf. Remark \ref{rkbeurlingconvolution}). We then study properties of its dual space. Inspired by various results on factorization of Banach and Fr\'{e}chet convolution algebras \cite{kisynski,Pet-v,r-s-t,voigt}, we introduce the Banach space $E'_{\ast}= L^{1}_{\check{\omega}}\ast E'$. The  convolution module structures of $E$ and $E'_{\ast}$ are crucial for achieving the main results of this article.

Our new distribution spaces are introduced in Section \ref{new spaces}. The test function space $\mathcal{D}_{E}$ consists of tempered distributions for which all partial derivatives belong to $E$. We first show that $\mathcal{D}_{E}$ is a Fr\'{e}chet space of smooth functions and actually the following inclusions between familiar test function spaces hold $\mathcal{S}(\mathbb{R}^{n})\hookrightarrow \mathcal{D}_{E}\hookrightarrow \mathcal O_C(\mathbb R^n)\hookrightarrow \mathcal{E}(\mathbb{R}^{n})$. The space $\mathcal D'_{E'_*}$ is defined as the strong dual of $\mathcal{D}_{E}$. It satisfies $\mathcal{E}'(\mathbb R^n)\hookrightarrow \mathcal{O}'_{C}(\mathbb{R}^{n})\rightarrow \mathcal {D}'_{E'_*}\hookrightarrow \mathcal{S}'(\mathbb{ R}^n)$, the second inclusion becomes dense when $E$ is reflexive. We study various structural and topological properties of $\mathcal{D}'_{E'_{\ast}}$ via the Schwartz parametrix method \cite{S}. In particular, it is proved that every $f\in\mathcal{D}'_{E'_{\ast}}$ is the finite sum of partial derivatives of elements of the Banach space $E'_{\ast}$. If $E$ is reflexive, we prove that $\mathcal{D}_{E}$ is an $FS^{\ast}$ space and $\mathcal{D}'_{E'_{\ast}}$ is a $DFS^{\ast}$ space \cite{komatsu67}, so that they are reflexive in this case. As examples we present in Section \ref{examples} the weighted spaces $\mathcal{D}'_{L^p_{\eta}}$, $1\leq p<\infty$, and the space of $\eta$-bounded distributions $\mathcal{B}'_{\eta}$, where $\eta$ is a polynomially bounded weight function. The article concludes with the study of convolution and multiplicative products on $\mathcal{D}'_{E'_{\ast}}$ in Section \ref{subsect convolution}.

\section{Notation}
\label{notation}

We use the standard notation from Schwartz distribution theory. In particular, if $K$ is a compact subset of $\mathbb{R}^{n}$, $\mathcal{D}_{K}$ is the Fr\'{e}chet space of smooth functions with ${\rm supp}\:\varphi\subseteq K$, provided with the family of seminorms

$$p_{j}(\varphi):=\underset{|\alpha|\leq j}{\sup_{x\in\mathbb{R}^{n}}}|\varphi^{(\alpha)}(x)|, \ \ \ j\in\mathbb{N}_{0},$$
where $\mathbb{N}_0=\mathbb{N}\cup\left\{0\right\}$. On the space $\mathcal{S}(\mathbb{R}^n)$, we use the family of norms

$$q_j(\varphi)=\underset{|\alpha|\leq j}{\sup_{x\in \mathbb{R}^n}}(1+|x|)^j|\varphi^{(\alpha)}(x)|, \ \ \ j\in\mathbb{N}_{0},$$
where $|x|$ denotes the Euclidean norm of $x$. The function $\check{g}$ denotes the reflection $\check{g}(x)=g(-x)$. Given $h\in\mathbb{R}^{n}$, we employ the notation $T_{h}$ for the translation operator, that is, $(T_{h}g)(x)=g(x+h)$. Naturally, the translation and reflection operations are well-defined for distributions as well. A subspace $Y\subset\mathcal{D}'(\mathbb{R}^{n})$ is called translation-invariant if $T_{h}(Y)= Y$ for all $h\in\mathbb{R}^{n}$.

The symbol $``\hookrightarrow"$ in the expression $X\hookrightarrow Y$ means dense and continuous linear embedding between topological vector spaces, while $``\hookrightarrow"$ in the expression $X\to Y$ just means a continuous embedding. Dual spaces $X'$ are always equipped with the strong dual topology.  Finally, if $Z\subseteq X'$, then $\sigma(X,Z)$ stands for the weak topology induced by $Z$ on $X$; in particular, $\sigma(X',X)$ is the weak$^\ast$ topology.

\section{On a class of translation-invariant Banach spaces}
\label{translation-invariant}

We are interested in the class of Banach spaces $E$ of distributions that satisfy the following three properties:
\begin{itemize}
    \item[ $(a)'$] $\mathcal{D}(\mathbb{R}^n)\hookrightarrow E\hookrightarrow \mathcal{D}'(\mathbb{R}^n)$.
    \item[$(b)'$] $T_h: E\to E$ for every $h\in\mathbb{R}^{n}$ (i.e., $E$ is translation-invariant).
		
        \item[$(c)'$] For any $g\in E$, there are $M=M_g>0$ and $\tau=\tau_{g}\geq 0$ such that
       $$
         \|T_{h} g\|_E\leq M(1+|h|)^\tau, \ \ \ \mbox{for all }  h\in \mathbb{R}^n.
       $$
\end{itemize}

We shall call any such Banach space satisfying the conditions $(a)'$, $(b)'$, and $(c)'$ a \emph{translation-invariant Banach space of tempered distributions}. It should be pointed out that $(a)'$ and $(b)'$ automatically imply that every translation operator $T_{h}:E\to E$ is continuous, as readily follows from the closed graph theorem. Our first important result tells us that $(a)'$, $(b)'$, and $(c)'$ may always be replaced by stronger conditions.

\begin{theorem}\label{pro:ime}
Let $E$ be a translation-invariant Banach space of tempered distributions. The following properties hold:
    \begin{itemize}
        \item [$(a)$] $\mathcal{S}(\mathbb{R}^n)\hookrightarrow E\hookrightarrow \mathcal{S}'(\mathbb{R}^n)$.
            \item[$(b)$] The mappings $\mathbb{R}^n\to E$ given by $h\mapsto T_hg$ are continuous for each $g\in E$.
        \item [$(c)$] There are absolute constants $M>0$ and $\tau\geq 0$ such that
         \[
    \|T_{h}g\|_E\leq M\|g\|_E (1+|h|)^\tau, \ \ \ \mbox{for all } g\in E \mbox{ and } h\in \mathbb{R}^n.
    \]
    \end{itemize}
\end{theorem}
\begin{proof}
Let us first prove $(c)$. Consider the following sets,
$$E_{j,\nu}=\{g\in E\,\,:\, \|T_{h}g\|_E\leq j (1+|h|)^\nu\mbox \,\,\mbox{for all }\mbox{   } h \in \mathbb{R}^n\}, \ \ \ j,\nu\in \mathbb{N}.$$
Because of $(c)'$, we have $E=\bigcup_{j,\nu\in \mathbb{N}}E_{j,\nu}$. Baire's theorem implies that one of the sets $E_{j_{0},\nu_{0}}$ contains a ball $\{f\in E:\: \|f-u\|_E<r\}$. If $g\in E$ is such that $\|g\|_E<r$, then

$$\|T_{h}g\|_E\leq\|T_{h}g+T_{h}u\|_E+\|T_{h}u\|_E\leq 2j_0(1+|h|)^{\nu_0}$$ for all $h \in \mathbb{R}^n$. So, for arbitrary $g\in E$, we get
$\|T_{h}g\|_E<(4 j_0/r) (1+|h|)^{\nu_0}\|g\|_E.$

The property $(b)$ follows easily from $(a)'$, $(b)'$ and $(c)$.

Let us now show $(a)$.  We first prove the embedding $\mathcal{S}(\mathbb{R}^{n})\hookrightarrow E$. Since $\mathcal D(\mathbb R^n)\hookrightarrow \mathcal S(\mathbb R^n),$  it is enough to prove that
$\mathcal{S}(\mathbb{R}^{n})\subset E$ and the continuity of the inclusion mapping. Let $\varphi\in \mathcal{S}(\mathbb{R}^{n})$. We use a special partition of unity:
$$
1=\sum_{m\in\mathbb{Z}^{n}} \psi(x-m), \ \ \ \psi\in\mathcal{D}_{[-1,1]^{n}} .
$$
Hence, we get the representation $\varphi(x)=\sum_{m\in\mathbb{Z}^{n}}\psi(x-m)\varphi(x)$. We estimate each term in this sum. Because of $(c)$,
\begin{equation}
\label{eq:1.1}
\|\varphi\: T_{-m}\psi\|_E\leq
\frac{M}{(1+|m|)^{n+1}}\|(1+|m|)^{n+\tau+1}\psi\:T_{m}\varphi\|_E ,
\end{equation}
where $|m|$ denotes Euclidean norm. We first prove that the multi-indexed sequence
$\left\{\rho_m\right\}_{m\in\mathbb{Z}^{n}}$ is bounded in $\mathcal D_{[-1,1]^{n}}$, where
\begin{equation}
\label{eq:1.2}
\rho_{m}=(1+|m|)^{n+\tau+1}\psi\:T_{m}\varphi.
\end{equation}
In fact, for any $j>n+\tau+1$, we have
\begin{equation}
\label{eq:1.3}
p_j(\rho_m)\leq p_{j}(\psi) \max_{|\alpha|\leq j}\sup_{|y-m|\leq 1}(1+|m|)^{j}|\varphi^{(j)}(y)|\leq M_{1} q_{j}(\varphi).
\end{equation}
By the assumption $(a)'$, the mapping $\mathcal{D}_{[-1,1]^{n}}\rightarrow E$ is continuous. So, there are $M_2>0$ and $j\in \mathbb{N}_0$, such that $\|\phi\|_E\leq M_2p_{j}(\phi)$, for every $\phi\in \mathcal{D}_{[-1,1]^n}$. We may assume that $j>n+\tau+1$. Therefore, by (\ref{eq:1.3}),

\begin{equation}
\label{eq:1.4}\|\rho_m\|_E\leq M_{1}M_{2}q_{j}(\varphi),  \   \  \  \mbox{for all } m\in\mathbb{Z}^{n}.
\end{equation}
Next, let $F(r)$ be the lattice counting function of points with integer coordinates inside the $n$-dimensional Euclidean closed ball of radius $r$. It is well known that $F(r)$ has asymptotics
$$F(r)=\sum_{m_{1}^{2}+\dots+m_{n}^{2}\leq r^{2}}1\sim \frac{\pi^{n/2}r^{n}}{\Gamma(n/2+1)}\:, \ \ \ r\to\infty.$$
In view of (\ref{eq:1.1}), (\ref{eq:1.2}), and (\ref{eq:1.4}), we obtain
\begin{equation}
\label{eq:1.5}
\|\sum_{N'<|m|\leq N}\varphi\: T_{-m}\psi\|_E
\leq M_{3}q_{j}(\varphi) \int_{N'}^{N}\frac{dF(r)}{(1+r)^{n+1}}\leq \frac{M_{4}q_{j}(\varphi)}{N'+1}
\end{equation}
and thus $\left\{\sum_{|m|\leq N} \varphi T_{-m}\psi\right\}_{N=0}^{\infty}$ is a Cauchy sequence in $E$  whose limit is $\varphi\in E.$ Taking $N'=0$ and $N\to\infty$ in (\ref{eq:1.5}), we get
$
\|\varphi\|_E
\leq M_{4}q_{j}(\varphi),$ for all $\varphi\in\mathcal{S}(\mathbb{R}^{n})$. The continuity of $\mathcal{S}(\mathbb{R}^{n})\hookrightarrow E$ has been established.

We now address  $E\subset\mathcal{S}'(\mathbb{R}^{n})$ and the continuity of the inclusion mapping. Let $g\in E$. Due to Schwartz' characterization of $\mathcal{S}'(\mathbb{R}^{n})$ \cite[Thm. VI, p. 239]{S}: $g$ belongs to $\mathcal{S}'(\mathbb{R}^{n})$ if and only if $g\ast \varphi$ is a function of at most polynomial growth for each $\varphi\in\mathcal{D}(\mathbb{R}^{n})$. Let $B$ be a bounded set in $\mathcal{D}(\mathbb{R}^{n})$. The embedding $E\hookrightarrow\mathcal{D}'(\mathbb{R}^{n})$ yields the existence of a constant $M_{5}=M_{5}(B)$ such that $|\left\langle g,\check{\phi}\right\rangle|\leq M_{5}||g||_{E}$ for all $ g\in E $ and $\phi\in B$. Therefore, by $(c)$,
\begin{equation}\label{eq:2}
|(g\ast \phi)(h)|\leq M_{5}||T_{h}g||_{E}\leq M_{5}M||g||_{E}(1+|h|)^{\tau},
\end{equation}
for all  $g\in E$, $\phi\in B,$ and  $h\in\mathbb{R}^{n}.$
This shows $E\subset \mathcal{S}'(\mathbb{R}^{n})$. The continuity of the inclusion mapping would follow if we show that the unit ball of $E$ is weakly bounded in $\mathcal{S}'(\mathbb{R}^{n})$. Fix $\varphi\in \mathcal{S}(\mathbb{R}^{n})$ and write again $\varphi(x)=\sum_{m\in\mathbb{Z}^{n}}\psi(x-m)\varphi(x)$, where $\psi$ is the partition of the unity used above. We use $\rho_{m}\in\mathcal{D}_{[-1,1]^n}$ as in (\ref{eq:1.2}). Taking (\ref{eq:2}) into account and the fact that $B=\left\{\check{\rho}_{m}:\: m\in\mathbb{Z}^{n}\right\}$ is a bounded subset of $\mathcal{D}(\mathbb{R}^{n})$ (cf. (\ref{eq:1.3})), we obtain, for all $g\in E$,
\begin{align*}
|\left\langle g,\varphi\right\rangle|&\leq \lim_{N\to\infty}\sum_{|m|\leq N}\frac{ |(g\ast\check{ \rho}_{m})(m)|}{(1+|m|)^{n+\tau+1}}
\\
&
\leq M_{5}M||g||_{E}\lim_{N\to\infty}\sum_{|m|\leq N}\frac{1}{(1+|m|)^{n+1}}\leq M_{6}||g||_{E}.
\end{align*}
Finally, the density of $E$ in $\mathcal{S}'(\mathbb{R}^{n})$ follows from the dense inclusion $\mathcal{S}(\mathbb{R}^{n})\hookrightarrow E$. The proof of $(a)$ is complete.
\end{proof}

Observe that condition $(c)$ gives us the order of growth in $h$ of the norms $||T_{h}||_{L(E)}$, where as usual $L(E)$ is the Banach algebra of continuous linear operators on $E$.

\begin{definition}\label{def:1} Let $E$ be a translation-invariant Banach space of tempered distributions. The growth function of the translation group is defined as
$$
\omega(h):= ||T_{-h}||_{L(E)}.
$$

\end{definition}

\smallskip

From now on, we shall \emph{always} assume that $E$ is a translation-invariant Banach space of tempered distributions with growth function $\omega$. It is clear that $\omega$ is measurable, $\omega(0)=1$, the function $\log \omega$ is subadditive, and, by $(c)$, it satisfies the estimate
\begin{equation}
\label{eqw}
\omega(h)\leq M (1+|h|)^{\tau}, \  \  \ h\in\mathbb{R}^{n}.
\end{equation}

We now study various properties of $E$. We start with the convolution.

\begin{lemma}\label{lemma:ime2}
The convolution mapping $(\varphi,\psi)\in \mathcal{S}(\mathbb{R}^n)\times \mathcal{S}(\mathbb{R}^n)\rightarrow \varphi\ast\psi \in \mathcal{S}(\mathbb{R}^n) $ extends to a continuous bilinear mapping $\mathcal{S}(\mathbb{R}^n)\times E\rightarrow E$. Furthermore, the following estimate holds
\begin{equation}\label{1}
\|\varphi\ast g\|_E\leq \|g\|_E\int_{\mathbb{R}^n}|\varphi(x)|\:\omega(x)dx.
\end{equation}
\end{lemma}
\begin{proof}
Given $\varphi, \psi \in \mathcal{D}(\mathbb{R}^{n})$ we can view $(\varphi\ast \psi)(x)=\int_{\rm{supp}\:\varphi}\varphi(t)\psi(x-t)dt$ as an integral in the Fr\'{e}chet space $\mathcal{S}(\mathbb{R}^n)$. Since $\mathcal{S}(\mathbb{R}^n)\hookrightarrow E$, the Riemann sums of this integral converge to $\varphi*\psi$ in the Banach space $E$. We have $\|\sum_j(t_{j+1}-t_{j})\varphi(t_j)T_{-t_{j}}\psi\|_{E}\leq \|\psi\|_E\sum_j(t_{j+1}-t_{j})|\varphi(t_j)|\:\omega(t_{j})$. Passing to the limit of the Riemann sums, we obtain $\|\varphi*\psi\|_E\leq \|\psi\|_E\int_{\mathbb{R}^n}|\varphi(t)|\:\omega(t)dt$. By using a standard density argument, one obtains the desired extension and (\ref{1}).
\end{proof}
The convolution of elements of $E$ can actually be performed with more general functions. Let $L^{1}_{\omega}$ be the Beurling algebra \cite{beurling,reiter} with weight $\omega$, i.e., the Banach algebra of measurable functions $u$ such that $||u||_{1,\omega}:=\int_{\mathbb{R}^{n}}|u(x)| \:\omega(x) dx<\infty$. Since $\mathcal{S}(\mathbb{R}^{n})$ is dense in this Beurling algebra, we obtain the ensuing proposition, a corollary of Lemma \ref{lemma:ime2}.

\begin{proposition}
\label{cor:Beurling Algebra}
The convolution extends to a mapping  $\ast:L^{1}_{\omega}\times E\rightarrow E$ and $E$ becomes a Banach module over the Beurling algebra $L^{1}_{\omega}$, i.e., $ ||u\ast g||_{E}\leq ||u||_{1,\omega}||g||_{E}$.
\end{proposition}

We shall denote this extension simply by $u*g=g*u$ whenever $u \in L^{1}_{\omega}$ and $g\in E$.
We call  $L^{1}_{\omega}$ the \emph{associated Beurling algebra} to $E$.

Lemma \ref{lemma:ime2} also allows us to consider approximations in $E$ by smoothing with test functions.
\begin{corollary}\label{molifaer}
    Let $g\in E$ and $\varphi\in \mathcal{S}(\mathbb{R}^n)$. Set $ \varphi_{\varepsilon}(x)=\varepsilon^{-n}\varphi \left(x/ \varepsilon \right)$.  Then, $$ \lim_{\varepsilon\to 0^+} \|cg-\varphi_\varepsilon*g\|_E=0,$$ where $c=\int_{\mathbb{R}^n}\varphi(x)dx$.
\end{corollary}

\begin{proof}

We first consider the case when $\varphi\in \mathcal{D}(\mathbb{R}^n)$. As in the proof of Lemma \ref{lemma:ime2}, for $g\in \mathcal{S}(\mathbb{R}^{n})$ we can view $g\ast\varphi_{\varepsilon}= \int_{\mathbb{R}^{n}}(T_{-y}g)\: \varphi_\varepsilon(y)dy$, an $E$-valued integral. Thus, if $\varepsilon<1$,
\begin{align*}
\|cg-\varphi_\varepsilon*g\|&=\left\|\int_{\mathbb{R}^n} (g-T_{-y}g)\frac{1}{\varepsilon^n}\varphi\left(\frac{y}{\varepsilon}\right)dy\right\|_E
\\
&\leq \sup_{t\in\mathrm{supp}\:\varphi }\left\|g-T_{-\varepsilon t}g\right\| _E \int_{\rm {supp\varphi}}\left|\varphi(t) \right|dt.
\end{align*}
Due to the density of $\mathcal{S}(\mathbb{R}^{n})\hookrightarrow E$, the above inequality remains true for $g\in E$. Hence, in view of condition $(b)$, this gives the result when $g\in E$ and $\varphi\in\mathcal{D}(\mathbb{R}^{n})$.
In the general case, let $\varphi\in \mathcal{S}(\mathbb{R}^n)$ and let $\{\psi_{j}\}_{j=1}^\infty\in \mathcal{D}(\mathbb{R}^{n})$ be a sequence such that $\psi_{j}\rightarrow\varphi$ in $\mathcal{S}(\mathbb{R}^{n})$.
By Lemma \ref{lemma:ime2} and (\ref{eqw}), we have for $\varepsilon <1$,
\begin{equation*}
\|(\psi_j)_\varepsilon*g-\varphi_\varepsilon*g\|_E
\leq M\|g\|_E\int_{\mathbb{R}^n}(1+|x|)^\tau|\psi_j(x)-\varphi(x)|dx,
\end{equation*}
whence the result follows because $\int_{\mathbb{R}^{n}}\psi_{j}(x)dx\to c$. \end{proof}

We now study the dual space of $E$.
\begin{proposition}\label{prop3.2}
The space $E'$ satisfies

 \begin{itemize}
        \item [$(a)''$] $\mathcal{S}(\mathbb{R}^n)\to E'\hookrightarrow \mathcal{S}'(\mathbb{R}^n)$.
        \item [$(b)''$] The mappings $\mathbb{R}^n\to E'$ given by $h\mapsto T_hf$ are continuous for the weak$^{\ast}$ topology.
    \end{itemize}
Moreover, the property $(c)$ from Theorem \ref{pro:ime} holds true when $E$ is replaced by $E'$.
\end{proposition}
\begin{proof}
It follows from $(a)$ that $\mathcal{S}(\mathbb{R}^n)\to E'\hookrightarrow \mathcal{S}'(\mathbb{R}^n)$. Given $f\in E'$ and $\varphi \in \mathcal{S}(\mathbb{R}^n)$,
$$
|\langle T_{h}f,\varphi\rangle|=| \langle f, T_{-h}\varphi\rangle|\leq \omega(h)||f||_{E'}||\varphi||_{E}\leq M\|f\|_{E'}\|\varphi\|_E(1+|h|)^{\tau}.
$$
Since $\mathcal{S}(\mathbb{R}^n)$ is dense in $E$, $T_{h}f\in E'$ and
$\|T_{h}f\|_{E'}\leq  M\|f\|_{E'}(1+|h|)^\tau.
$
By (b) applied to $E$,
$\lim_{h\rightarrow h_0}\left\langle T_{h}f-T_{h_{0}}f,g\right\rangle=\left\langle f,\lim_{h\rightarrow h_0}(T_{-h}g-T_{-h_{0}}g)\right\rangle=0,$ for each $g\in E$.

\end{proof}

We can also associate a Beurling algebra to $E'$. Set
$$\check{\omega}(h):=||T_{-h}||_{L(E')}=||T^{\top}_{h}||_{L(E')}=\omega(-h).$$

The associated Beurling algebra to the dual space $E'$ is $L^{1}_{\check{\omega}}$. We define the convolution $u\ast f=f\ast u$ of $f\in E'$ and $u\in L^{1}_{\check{\omega}}$ via transposition:
\begin{equation}
\label{eq:convolution}
\left\langle u\ast f,g \right\rangle:= \left\langle f,\check {u}\ast g\right\rangle, \ \ \ g\in E.
\end{equation}
In view of Proposition \ref{cor:Beurling Algebra}, this convolution is well-defined because $\check {u}\in L^{1}_{\omega}$.

\begin{corollary} \label{cor3.3} We have $ ||u\ast f||_{E'}\leq ||u||_{1,\check{\omega}}||f||_{E'}$ and thus $E'$ is a Banach module over the Beurling algebra $L^{1}_{\check{\omega}}$. In addition, if $\varphi_{\varepsilon}$ and $c$ are as in Corollary \ref{molifaer}, then $\varphi_{\varepsilon}\ast f \to c f$ as $\varepsilon \to 0$ weakly$^{\ast}$ in $E'$ for each fixed $f\in E'$.
\end{corollary}

In general the embedding $\mathcal{S}(\mathbb{R}^{n})\to E'$ is not dense (consider for instance $E= L^{1}$). However, $E'$ inheres the three properties $(a)$, $(b)$, and $(c)$ whenever $E$ is reflexive.

\begin{proposition} \label{prop3.3} If $E$ is reflexive, then its dual space $E'$ is also a  translation-invariant Banach space of tempered distributions.
\end{proposition}
\begin{proof} By Proposition \ref{prop3.2}, it is enough to see that $\mathcal{S}(\mathbb{R}^{n})$ is dense in $E'$ (the property $(b)'$ trivially holds for $E'$). But if $g\in E''=E$ is such that $\left\langle g,\varphi\right\rangle=0$ for all $\varphi\in\mathcal{S}(\mathbb{R}^{n})$, then the property $(a)$ of $E$ implies that $g=0$. 
Thus, $\mathcal{S}(\mathbb{R}^{n})$ is dense in $E'$.
\end{proof}

The fact that property $(b)$ fails for $E'$ in the non-reflexive case ($E=L^{1}$ is again an example) causes various difficulties when dealing with this space. We will often work with a certain closed subspace of $E'$ rather than with $E'$ itself. We denote the linear span of a set $A$ as $\operatorname*{span} (A)$.

\begin{definition}\label{goodspace} The Banach space $E'_{\ast}$ is defined as $E'_{\ast}=L^{1}_{\check{\omega}}\ast E'$.
\end{definition}

 That $E'_{\ast}$ is a closed linear subspace of $E'$ is a non-trivial fact. It follows from the celebrated Cohen-Hewitt factorization theorem \cite{kisynski}, which asserts in this case the equality
$L^{1}_{\check{\omega}}\ast E'=\overline{\operatorname*{span}}(L^{1}_{\check{\omega}}\ast E')$ because the Beurling algebra $L^{1}_{\check{\omega}}$ possesses a bounded approximation unity (e.g., $\{\varphi_{\varepsilon}\}_{\varepsilon\in (0,1)}$ such that $c=1$ with the notation of Corollary \ref{molifaer}). The space $E'_{\ast}$ will be of crucial importance throughout the rest of this work. It possesses richer properties than $E'$ with respect to the translation group, as stated in the next theorem. The proof of the ensuing result makes use of an important property of the Fr\'{e}chet algebra $\mathcal{S}(\mathbb{R}^{n})$. Miyazaki \cite[Lem. 1, p. 529]{Miyazaki} (cf. \cite{Pet-v,voigt})) has shown the factorization theorem $\mathcal{S}(\mathbb{R}^{n})=\mathcal{S}(\mathbb{R}^{n})\ast \mathcal{S}(\mathbb{R}^{n})$ (the related result $\mathcal{D}=\operatorname*{span}(\mathcal{D}\ast\mathcal{D})$ has been proved in \cite{r-s-t}).

\begin{theorem}
\label{thgoodspace} The space $E'_{\ast}$ has the properties $(a)''$, $(b)$, and $(c)$. It is a Banach module over the Beurling algebra $L^{1}_{\check{\omega}}$. If $\varphi_{\varepsilon}$ and $c$ are as in Corollary \ref{molifaer}, then, for each $f\in E'_{\ast}$,
\begin{equation}
\label{eqapprox}
\lim_{\varepsilon\to0^{+}}\|cf-\varphi_{\varepsilon}\ast f\|_{E'}=0.
\end{equation}
Furthermore, if $E$ is reflexive, then $E'_{\ast}=E'$.
\end{theorem}
\begin{proof}
For $(a)''$,
$\mathcal{S}(\mathbb{R}^{n})=\mathcal{S}(\mathbb{R}^{n})\ast \mathcal{S}(\mathbb{R}^{n})\subset L^{1}_{\check{\omega}}\ast E' = E'_{\ast},$
 whence the assertion follows. The property $(c)$ for $E'_{\ast}$ directly follows from Proposition \ref{prop3.2}. Observe that $\|T_{h}(u\ast f)- u\ast f\|_{E'}\leq \| f \|_{E'} \| T_{h}u-u \|_{1,\check{\omega}} \to 0$ as $h\to 0$ for $u\in L^{1}_{\check{\omega}}$ and $f\in E'$. The property (b) on $E'_{\ast}$ then follows. Likewise, the approximation property (\ref{eqapprox}) is easily established. Finally, if $E$ is reflexive we have, by Proposition \ref{prop3.3}, that $\mathcal{S}(\mathbb{R}^{n})$ is dense in $E'$; since $\mathcal{S}(\mathbb{R}^{n})\subset E'_{\ast}$ and  $E'_{\ast}$ is closed, we must have $E'_{\ast}=E'$.
\end{proof}

We point out factorization properties of the Banach modules $E$ and $E'_{\ast}$ which also follow from the Cohen-Hewitt factorization theorem.

\begin{proposition}\label{factorization} The factorizations $E=L^{1}_{\omega}\ast E$ and $E'_{\ast}= L^{1}_{\check{\omega}} \ast E'_{\ast}$ hold.
\end{proposition}
\begin{proof}
The Cohen-Hewitt factorization theorem yields $L^{1}_{\omega}\ast E=\overline{\operatorname*{span}(L^{1}_{\omega}\ast E)}$ and $L^{1}_{\check{\omega}} \ast E'_{\ast}=\overline{\operatorname*{span}(L^{1}_{\check{\omega}} \ast E'_{\ast})}$. By Corollary \ref{molifaer} and Theorem \ref{thgoodspace}, $E=\overline{\mathcal{S}(\mathbb{R}^{n})\ast E}\subseteq \overline{\operatorname*{span}(L^{1}_{\omega}\ast E)}=L^{1}_{\omega}\ast E$ and $E'_{\ast}=\overline{\mathcal{S}(\mathbb{R}^{n})\ast E'_{\ast}}\subseteq \overline{\operatorname*{span}(L^{1}_{\check{\omega}}\ast E'_{\ast})}=L^{1}_{\check{\omega}}\ast E'_{\ast}$, that is, $E=L^{1}_{\omega}\ast E$  and $E'_{\ast}= L^{1}_{\check{\omega}} \ast E'_{\ast}$.
\end{proof} 
We now characterize $E'_{\ast}$ by showing that it is the biggest subspace of $E'$ where the property $(b)$ holds.

\begin{proposition}\label{prop3.11} We have that $E'_{\ast}=\{f\in E': \lim_{h\to 0}\|T_{h}f-f\|_{E'}=0\}.$
\end{proposition}
\begin{proof} Call momentarily $X=\{f\in E': \lim_{h\to 0}\|T_{h}f-f\|_{E'}=0\}$, it is clearly a closed subspace of $E'$. By the approximation property (\ref{eqapprox}), it is enough to prove that $\mathcal{D}(\mathbb{R}^{n})\ast E'$ is dense in $X$. For this, we will show that if $\varphi\in\mathcal{D}(\mathbb{R}^{n})$ is positive and $\int_{\mathbb{R}^{n}}\varphi(y)dy=1$, then $\lim_{\varepsilon\to0}\|f\ast \varphi_{\varepsilon}-f\|_{E'}=0$, for $f\in X$. We apply a similar argument to that used in the proof of Corollary \ref{molifaer}. Take $\phi\in \mathcal{D}(\mathbb{R}^{n})$, then
$$| \langle f\ast \varphi_{\varepsilon}-f,\phi\rangle |=\left| \left\langle f,\int_{\mathbb{R}^{n}}\varphi(y)(T_{\varepsilon y}\phi-\phi) dy\right\rangle \right|\leq \|\phi\|_{E}\sup_{y\in \operatorname*{supp}\varphi} \|T_{-\varepsilon y}f-f\|_{E'} ,$$
which shows the claim.
\end{proof}

In view of property $(b)''$ from Proposition \ref{prop3.2},  we can naturally define a convolution mapping $E'\times \check{E}
\to C(\mathbb{R}^{n})$, where $\check{E}=\{g\in \mathcal{S}'(\mathbb{R}^{n}):\: \check{g}\in E\}$ with norm $\|g\|_{\check{E}}:=\|\check{g}\|_{E}$. We give a simple proposition that describes the mapping properties of this convolution.
As usual, $L^{\infty}_{\omega}$, the dual of the Beurling algebra $L^{1}_{\omega}$, is the Banach space of all measurable functions satisfying
$$
||u||_{\infty,\omega}=\operatorname*{ess}\sup_{x\in\mathbb{R}^{n}} \frac{|u(x)|}{\omega(x)}<\infty.
$$
We need the following two closed subspaces of $L^{\infty}_{\omega}$,
\begin{equation}
\label{UC}
UC_{\omega}:= \left\{u\in L^{\infty}_{\omega}: \lim_{h\to0}||T_{h}u-u||_{\infty,\omega}=0 \right\}
\end{equation}
and
\begin{equation}
\label{C}
C_{\omega}:= \left\{u\in C(\mathbb{R}^{n}): \lim_{|x|\to\infty}\frac{u(x)}{\omega(x)}=0 \right\}
\end{equation}
\begin{proposition}
\label{convolution E E'}
$E'\ast \check{E} \subseteq UC_{\omega}$ and $\ast:E'\times \check{E} \to UC_{\omega}$ is continuous. If $E$ is reflexive, then $E'\ast \check{E} \subset C_{\omega}$.
\end{proposition}
\begin{proof} The first assertion follows at once from the property $(b)''$. If $E'$ is reflexive, Proposition \ref{prop3.3} gives $\mathcal{S}(\mathbb{R}^{n})\hookrightarrow E'$. Thus, $\mathcal{S}(\mathbb{R}^{n})=\mathcal{S}(\mathbb{R}^{n})\ast\mathcal{S}(\mathbb{R}^{n})$ is dense in the closure of $\operatorname*{span}(E'\ast \check{E})$ with respect to the norm $\|\:\:\|_{\infty,\omega}$. Since $\mathcal{S}(\mathbb{R}^{n})$ is obviously dense in
$C_{\omega}$, we obtain $E'\ast \check{E} \subseteq C_{\omega}$. 
\end{proof}

We end this section with the following remark.
\begin{remark}\label{rkbeurlingconvolution} The properties from Lemma \ref{lemma:ime2} and Corollary \ref{molifaer} essentially characterize the class of translation-invariant Banach spaces of tempered distributions in the following sense. Let $X$ be a Banach space that satisfies the condition $(a)'$ and let $\eta:\mathbb{R}^{n}\to (0,\infty)$ be a measurable function such that $\log \eta$ is subadditive, $\eta(0)=1$, and $\eta$ is polynomially bounded. Assume that $\|\varphi\ast g \|_{X}\leq \|\varphi\|_{1,\eta}\|g\|_{X}$ for all $\varphi\in\mathcal{D}(\mathbb{R}^{n})$ and $g\in X$. The density of $\mathcal{D}(\mathbb{R}^{n})$ in $L^{1}_{\eta}$ automatically guarantees that $X$ becomes a Banach convolution module over the Beurling algebra $L^{1}_{\eta}$ and the convolution obviously satisfies $T_{h}(u\ast g)=(T_{h} u)\ast g=u\ast (T_{h}g)$. If we additionally assume that $L^{1}_{\eta}$ possesses a bounded approximation of the unity for $X$, that is, there is a sequence $\{e_{j}\}_{j=0}^{\infty}\subset L^{1}_{\eta}$ such that $\sup_{j}\|e_{j}\|_{1,\eta}=M<\infty$, $\lim_{j}\|e_{j}\ast u- u\|_{1,\eta}=0$, and $\lim_{j}\|e_{j}\ast g- g\|_{X}$ for all $u\in L^{1}_{\eta}$, $g\in X$, then the Cohen-Hewitt theorem yields the factorization $X=L^{1}_{\eta}\ast X$. The latter factorization property implies that $X$ satisfies the conditions $(b)'$ and $(c)'$. In addition, its weight function $\omega$ satisfies $\omega(x)\leq M\eta(x)$.
\end{remark}

%
\section{New distribution spaces}\label{new spaces}
In this section we construct and study test function and distribution spaces associated to  translation-invariant Banach spaces. We recall that throughout the rest of the paper $E$ stands for a translation-invariant Banach space of tempered distributions whose growth function of its translation group is $\omega$ (cf. Definition \ref{def:1}). The Banach space $E'_{\ast}\subseteq E'$ was introduced in Definition \ref{goodspace}.

\subsection{The test function space $\mathcal{D}_{E}$}\label{test space}
We begin by constructing our space of test functions.
Let $\mathcal{D}_E$ be the subspace of tempered distributions $\varphi\in \mathcal{S}'(\mathbb{R}^n)$ such that $\varphi^{(\alpha)}\in E$ for all $\alpha\in \mathbb{N}_{0}^n$. We topologize $\mathcal{D}_E$ by means of the family of norms
\begin{equation}\label{eq4}
\|\varphi\|_{E,N}:=\max_{|\alpha| \le N}\|\varphi^{(\alpha)}\|_E.
\end{equation}
\begin{proposition}\label{prop:ime3}
$\mathcal{D}_E$ is a Fr\'{e}chet space and $\mathcal{S}(\mathbb{R}^n)\hookrightarrow\mathcal{D}_E\hookrightarrow E\hookrightarrow \mathcal{S}'(\mathbb{R}^n)$. Moreover, $\mathcal{D}_E$ is a Fr\'{e}chet module over the Beurling algebra $L^{1}_{\omega}$, namely,
\begin{equation}\label{eq5}||u\ast \varphi||_{E,N}\leq ||u||_{1,\omega}|| \varphi||_{E,N},\ \ \ N\in\mathbb{N}_{0}.
\end{equation}
\end{proposition}
\begin{proof}
$\mathcal{D}_{E}$ is a Fr\'{e}chet space as a countable intersection of Banach spaces and $\mathcal{S}(\mathbb{R}^n)\subseteq\mathcal{D}_E\hookrightarrow E\hookrightarrow \mathcal{S}'(\mathbb{R}^n)$. The relation (\ref{eq5}) follows from Proposition \ref{cor:Beurling Algebra} and the definition of the norms (\ref{eq4}). It remains to show the density of the embedding $\mathcal{S}(\mathbb{R}^n)\hookrightarrow\mathcal{D}_E$. Let $\varphi\in\mathcal{D}_{E}$ and fix $N\in\mathbb{N}$. Find a sequence $\left\{\psi_{j}\right\}_{j=1}^{\infty}$ of functions from $\mathcal{S}(\mathbb{R}^{n})$ such that $||\varphi-\psi_{j}||_{E}\leq j^{-N-1}$, for all $j$. Pick then $\phi\in\mathcal{S}(\mathbb{R}^{n})$ such that $\int_{\mathbb{R}^{n}}\phi(x)dx=1$ and set $\phi_{j}(x)=j^{n}\phi(j x)$. We show that $\psi_{j}\ast \phi_{j}\to \varphi$ with respect to the norm $||\: \:||_{E,N}$; indeed, by Corollary \ref{molifaer} and Proposition \ref{cor:Beurling Algebra},
\begin{equation*}
\limsup_{j\to\infty}||\varphi-\psi_{j}\ast \phi_{j}||_{E,N}\leq  \limsup_{j\to\infty} j^{N}||\varphi-\psi_{j}||_{E} \max_{|\alpha|\leq N}\int_{\mathbb{R}^{n}} |\phi^{(\alpha)}(x)|\: \omega(x/j)d x
=0.
\end{equation*} 
\end{proof}

It turns out that all elements of our test function space $\mathcal{D}_{E}$ are smooth functions. We need a lemma in order to establish this fact.

\begin{lemma}
\label{lemma4.2} Let $K\subset \mathbb{R}^{n}$ be compact. There is a positive integer $j$ such that $\mathcal{D}_{K}^{j}\subset E\cap E'_{\ast}$ and the inclusion mappings $\mathcal{D}_{K}^{j}\to E$ and $\mathcal{D}_{K}^{j}\to E'_{\ast}$ are continuous.
\end {lemma}
\begin{proof} We may of course assume that $K$ has non-empty interior. Let $\sigma>0$ and set $K_{\sigma}=K+\{x\in\mathbb{R}^{n}: |x|\leq \sigma\}$. Since $\mathcal{D}(\mathbb{R}^{n}) \hookrightarrow E$ and $\mathcal{D}(\mathbb{R}^{n}) \to E'_{\ast}$ are continuous, there is $j=j_{K_{
\sigma}}\in\mathbb{N}$ such that
\begin{equation}
\label{eqlemma}
||\varphi||_{E}\leq M_{K_{\sigma}} p_{j}(\varphi) \  \  \   \mbox{and} \  \  \ ||\varphi||_{E'}\leq M_{K_{\sigma}} p_{j}(\varphi)
\end{equation}
for every $\varphi\in \mathcal{D}_{K_{\sigma}}$. Using a regularization argument, Corollary \ref{molifaer}, and Theorem \ref{thgoodspace}, we convince ourselves that (\ref{eqlemma}) remains valid for all $\varphi\in\mathcal{D}_{K}^{j}$.
\end{proof}

We can now show that $\mathcal{D}_{E}\hookrightarrow \mathcal{E}(\mathbb{R}^{n})$. More generally \cite{horvath,S}, let $\mathcal{O}_{C}(\mathbb{R}^{n})$ be the test function space corresponding to the space $\mathcal{O}'_{C}(\mathbb{R}^{n})$ of convolutors of $\mathcal{S}'(\mathbb{R}^{n})$, that is, $\varphi\in\mathcal{O}_{C}(\mathbb{R}^{n})$ if there is $k\in\mathbb{N}$ such that $|\varphi^{(\alpha)}(x)|\leq M_{\alpha}(1+|x|)^{k}$, for all $\alpha$. It is topologized by a canonical inductive limit topology as in \cite{horvath}. The spaces of continuous functions $UC_{\omega}$ and $C_{\omega}$ were introduced in (\ref{UC}) and (\ref{C}). We have,

\begin{proposition}
\label{smooth prop} The embedding $\mathcal{D}_{E}\hookrightarrow \mathcal{O}_{C}(\mathbb{R}^{n})$ holds. Furthermore, the partial derivatives of every $\varphi\in\mathcal{D}_{E}$ are elements of $C_{\check{\omega}}$, namely, they have decay
\begin{equation}
\label{eqgrowth}
\lim_{|x|\to\infty} \frac{\varphi^{(\alpha)}(x)}{\omega(-x)}=0, \  \  \  \ \alpha\in\mathbb{N}^{n}.
\end{equation}
\end{proposition}
\begin{proof}We will employ the powerful Schwartz parametrix method \cite{S}. Let $K$ be a compact symmetric neighborhood of 0 and find $\chi\in \mathcal{D}_K$ such that $\chi=1$ near 0. Consider the Laplace operator $\Delta$ on $\mathbb{R}^{n}$. Let $F_{l}$ be a fundamental solution of $\Delta^{l}$, i.e., $\Delta^l F_{l}=\delta$. Then,  $\triangle^l(\chi F_{l})-\delta=\varsigma_{l}\in \mathcal{D}(\mathbb{R}^n)$, so that the parametrix formula
\begin{equation}
\label{eq:6}
f=\triangle^l((\chi F_{l})\ast f)-\varsigma_{l}*f
\end{equation}
holds for every $f\in\mathcal{D}'(\mathbb{R}^{n})$. By Lemma \ref{lemma4.2}, one can find $j\in\mathbb{N}$ for which $\mathcal{D}_{K}^{j}\subset E\cap E'$. Let $\varphi\in \mathcal{D}_{E}$ .
Since there is a sufficiently large $l\in\mathbb{N}$ such that  $\chi F_{l}\in \mathcal{D}^{j}_{K}\subset E'$, we conclude from (\ref{eq:6}) and Proposition
\ref{convolution E E'} that for each $\alpha\in\mathbb{N}^{n}$ one has $\varphi^{(\alpha)}=((\chi F_{l})\ast(\Delta^{l}\varphi^{(\alpha)}))-\varsigma_{l}\ast\varphi^{(\alpha)}\in \check{E}'\ast E\subset UC_{\check{\omega}}$ and, by Proposition \ref{prop3.2} we actually obtain
\begin{align*}
|\varphi^{(\alpha)}(x)|&\leq \omega(-x)\left(  \|\check{\chi} \check{F_{l}}\|_{E'}\|\Delta^{l}\varphi^{(\alpha)}\|_{E}+\|\check{\varsigma_{l}}\|_{E'}\|\varphi^{(\alpha)}\|_{E}\right)\\
&
\leq M_{l} \omega(-x)\|\varphi\|_{E,2l+|\alpha|},
\end{align*}
which also shows the embedding $\mathcal{D}_{E}\hookrightarrow \mathcal{O}_{C}(\mathbb{R}^{n})$. Furthermore, if we set $\|\cdot\|_{\infty,\check{\omega},N}:=\max_{|\alpha|\leq N}\|\cdot\|_{\infty,\check{\omega}}$, $N=0,1,2,\dots$, the above estimates imply
\begin{equation}
\label{bounds}
\|\varphi\|_{\infty,\check{\omega},N}\leq M_{l}\|\varphi\|_{E,N+2l}\: , \  \  \ \varphi\in \mathcal{D}_{E}, \ N\in\mathbb{N}_{0}.
\end{equation}
In order to show (\ref{eqgrowth}), we make use of the density $\mathcal{D}(\mathbb{R}^{n})\hookrightarrow \mathcal{D}_{E}$. Fix $N$. Given $\varepsilon>0$, find $\rho \in \mathcal{D}(\mathbb{R}^{n})$ such that  $\|\varphi-\rho\|_{E,2l+N}< \varepsilon/M_{l}$. Choose also $\lambda>0$ so large that $\rho(x)=0$ for all $|x|\geq \lambda$. By (\ref{bounds}), we obtain that $|\varphi^{(\alpha)}(x)|<\varepsilon\omega(-x)$ for all $|x|\geq \lambda$ and $|\alpha|\leq N$.
\end{proof}
\begin{remark}
\label{remark inequality} For $u\in\mathcal{S}'(\mathbb{R}^{n})$ with $u^{(\alpha)}\in L^{1}_{\omega}$, $|\alpha|\leq N$, set 
$$\|u\|_{1,\omega,N}:=\max_{|\alpha|\leq N}\|u^{(\alpha)}\|_{1,\omega}
$$
and keep $l$ as above. Note that $\chi F_{l} \in\mathcal{D}^{j}_{K}\subset E$. If $\varphi \in \mathcal{S}(\mathbb{R}^{n})$, Proposition \ref{cor:Beurling Algebra} leads to
$$
\|\varphi^{(\alpha)}\|_{E}\leq  \|\chi F_{l}\|_{E}\|\Delta^{l}\varphi^{(\alpha)}\|_{1,\omega}+\|\varsigma_{l}\|_{E}\|\varphi^{(\alpha)}\|_{1,\omega},  \  \   \  \alpha\in \mathbb{N}_{0}^{n},
$$
namely, $E$-norm bounds
\begin{equation}\label{D omega inequality}
\|\varphi\|_{E,N}\leq M'_{l}\|\varphi\|_{1,\omega,N+2l}\ ,  \ \  \ \varphi\in\mathcal{S}(\mathbb{R}^{n}),\ N\in \mathbb{N}_{0}.
\end{equation}
The inequality (\ref{D omega inequality}) will be employed in Section \ref{subsect convolution} to study further properties of $\mathcal{D}_{E}$.
\end{remark}

\subsection{The distribution space $\mathcal{D}'_{E'_{
\ast}}$} \label{subsection DE}

We can now define our new distribution space. We denote by $\mathcal{D}'_{E'_{\ast}}$ the strong dual of $\mathcal{D}_{E}$. When $E$ is reflexive, we write $\mathcal{D}'_{E'}=\mathcal{D}'_{E'_{\ast}}$ in accordance with the last assertion of Theorem \ref{thgoodspace}. In view of Proposition \ref{prop:ime3} and Proposition \ref{smooth prop}, we have the (continuous) inclusions $\mathcal{O}'_{C}(\mathbb{R}^{n})\subset\mathcal{D}'_{E'_{\ast}}\subset \mathcal{S}'(\mathbb{R}^{n})$. In particular, every compactly supported distribution belongs to the space $\mathcal{D}'_{E'_{\ast}}$.

The notation $\mathcal{D}'_{E'_{\ast}}=(\mathcal{D}_{E})'$ is motivated by the next structural theorem, which characterizes the elements of this dual space in two ways, in terms of convolution averages and as sums of derivatives of elements of $E'_{\ast}$ (or $E'$). These characterizations play a fundamental role in our further considerations.

\begin{theorem}\label{karak}
Let $f\in\mathcal{D}'(\mathbb{R}^n)$. The following statements are equivalent:
\begin{itemize}
\item [$(i)$] $f\in \mathcal{D}_{E'_{\ast}}'$.
\item [$(ii)$] $f*\psi\in E'$ for all $\psi\in \mathcal{D}(\mathbb{R}^n)$.
\item [$(iii)$] $f*\psi\in E'_{\ast}$ for all $\psi\in \mathcal{D}(\mathbb{R}^n)$.
\item [$(iv)$] $f$ can be expressed as $f=\sum_{|\beta|\leq N}g_{\beta}^{(\beta)}$, with $g_{\beta}\in E'$.
\item [$(v)$] There are $f_{\alpha}\in E'_{\ast}\cap UC_{\omega}$ such that
\begin{equation}\label{eq:representation}
 f=\sum_{|\alpha|\leq N}f_{\alpha}^{(\alpha)}.
\end{equation}
\end{itemize}
Moreover, if $E$ is reflexive, we may choose $f_{\alpha}\in E'\cap C_{\omega}$.
\end{theorem}
\begin{remark}\label{r1}
One can replace $\mathcal{D}'(\mathbb{R}^n)$ and $\mathcal{D}(\mathbb{R}^n)$ by $\mathcal{S}'(\mathbb{R}^n)$ and $\mathcal{S}(\mathbb{R}^n)$ in the statement of Theorem \ref{karak}. It follows from Theorem \ref{karak}, since $E'\subset \mathcal{D}_{E'_{\ast}}'$,  that every element of $f\in E'$ can be expressed as a sum of partial derivatives of elements of $E'_{\ast}\cap UC_{\omega}$ (or $E'\cap C_{\omega}$ in the reflexive case).
\end{remark}
\begin{proof}
Clearly, $(v)\Rightarrow(i)$. We denote below $B_{E}=\{ \varphi\in\mathcal{D}(\mathbb{R}^n):\: \|\varphi\|_E\leq1\}$.

$(i)\Rightarrow (ii)$. Fix first $\psi\in\mathcal{D}(\mathbb{R}^n)$. By Proposition \ref{cor:Beurling Algebra}, the set $\check{\psi}\ast B_{E}=\{\check{\psi}*\varphi:\: \varphi \in B_{E}\}$ is bounded in $\mathcal{D}_{E}$.

Hence, $|\langle f*\psi,\varphi\rangle|=|\langle f,\check{\psi}*\varphi\rangle|<M_{\psi}$ for $\varphi\in B_{E}$.
So, $|\langle f*\psi,\varphi\rangle|<M_{\psi}\|\varphi\|_E,$ for all $\varphi\in\mathcal{D}(\mathbb{R}^{n}).$
Using the fact that $\mathcal{D}(\mathbb{R}^{n})$ is dense in $E$, the last inequality means that $f*\psi\in E'$, for every $\psi\in \mathcal{D}(\mathbb{R}^{n})$.

$(ii)\Rightarrow (iii)$. We use the factorization property of $\mathcal{D}(\mathbb{R}^{n})$ from \cite{r-s-t} to write $\psi=\psi_{1}\ast\phi_{1}+\psi_{2}\ast\phi_{2}+\dots+\psi_{N}\ast\phi_{N}\in \mathcal{D}(\mathbb{R}^{n})$ with $\psi_{j},\phi_{j}\in \mathcal{D}(\mathbb{R}^{n})$. From $(ii)$, we conclude $f\ast \psi= (f\ast\psi_{1})\ast\phi_{1}+\dots+(f\ast\psi_{N})\ast\phi_{N}\in \operatorname*{span}(E'\ast \mathcal{D}(\mathbb{R}^{n}))\subset E'_{\ast},$ for any $\psi\in \mathcal{D}(\mathbb{R}^{n})$.

$(iii)\Rightarrow(iv)$.  Let $\psi\in\mathcal{D}(\mathbb{R}^n)$ be arbitrary. Because $\langle f*\check{\varphi},\check{\psi}\rangle=\langle f*\psi,\varphi\rangle$ we get that the set $\{\langle f*\check{\varphi},\check{\psi}\rangle:\: \varphi\in B_{E}\}$ is bounded in $\mathbb{C}$. The Banach-Steinhaus theorem implies that $
\{ f*\check{\varphi}:\:\varphi\in B_{E}\}$
is an equicontinuous subset of $\mathcal{D}'(\mathbb{R}^n)$. Namely,
for any compact set $K\subset \mathbb{R}^n$ there exist $N=N_{K}\in \mathbb{N}_0$ and $M=M_{K}>0$ such that $|\langle f*\rho,\varphi\rangle|<Mp_{N}(\rho)$ for every $\varphi\in B_{E}$ and $\rho\in \mathcal{D}_K^{N}$. Hence, for all $\rho\in \mathcal{D}_K ^N$ we have $ f\ast \rho \in E'$.

Let $K$, $\chi\in \mathcal{D}_K$ and $F_{l}$ be as in the proof of Proposition \ref{smooth prop}. Then $\chi F_{l}\in \mathcal{D}_K^{N}$ for sufficiently large $l$ so that the parametrix formula (\ref{eq:6}) yields
$f\in\triangle^l(E')+E'\subseteq \mathcal{D}_{E'_{\ast}}'$. In particular, one obtains the representation
\begin{equation}
\label{eq:6.1} f=\sum_{|\beta|\leq 2l}g_{\beta}^{(\beta)},  \  \  \ g_{\beta}\in E'.
\end{equation}

$(iv)\Rightarrow(v)$. In order to improve the representation (\ref{eq:6.1}) to the one stated in $(v)$, we apply the parametrix method again to each $g_{\beta}\in E'$, $|\beta|\leq 2l$. Let $K$ be a compact symmetric set as above. By Lemma \ref{lemma4.2}, one can find $j=j_{K}$ such that $\mathcal{D}^{j}_{K}\subset E$. Choosing $l'$ so large that $\chi F_{l'}\in \mathcal{D}_K^{j}$, the parametrix formula (\ref{eq:6}) yields
\begin{equation}
\label{eq:6.2}
g_{\beta}=\sum_{|\nu|\leq 2l'} (f_{\beta,\nu})^{(\nu)},
\end{equation}
where each $f_{\beta,\nu}\in L^1_{\check{\omega}}\ast E'\subset E'_{\ast}$. Furthermore, each $f_{\beta,\nu}$ is of the form $f_{\beta,\nu}=g_{\beta}\ast \check{\varrho_{\nu}}$ with ${\varrho}_{\nu}\in \mathcal{D}^{j}_{K}\subset E$. By Proposition \ref{convolution E E'}, we have $f_{\beta,\nu}\in UC_{\omega}$ (resp., $C_{\omega}$ in the reflexive case).
\end{proof}

We have the following interesting corollary.

\begin{corollary}\label{convolution corollary} Let $L:\mathcal{D}(\mathbb{R}^{n})\rightarrow E'_{\sigma(E',E)}$ be a continuous linear mapping, where $\sigma(E',E)$ stands for the weak$^{\ast}$ topology on $E'$. If $L$ commutes with every translation, i.e.,
\begin{equation}
\label{eq:7}
L(T_{h}\varphi) = T_{h}(L(\varphi)), \ \ \ \mbox{ for all } h\in\mathbb{R}^{n} \mbox{ and } \varphi\in \mathcal{D}(\mathbb{R}^{n}),
\end{equation}
then, there exists $f\in \mathcal{D}_{E'_{\ast}}'$ such that $L$ is of the form
\begin{equation}
\label{eq:8}
L(\varphi)=f\ast \varphi, \ \ \ \varphi\in\mathcal{D}(\mathbb{R}^{n}).
\end{equation}
\end{corollary}
\begin{proof}
Since $E'_{\sigma(E',E)}\to\mathcal{D}'(\mathbb{R}^{n})$ is continuous, we obtain that $L: \mathcal{D}(\mathbb{R}^{n})\to \mathcal{D}'(\mathbb{R}^{n})$ is also a continuous linear mapping. Due to the fact that $L$ commutes with every translation, it follows from a well-known theorem (cf. \cite[Thm. 5.11.3, p. 332]{Sch}) that there exists $f\in\mathcal{D}'(\mathbb{R}^{n})$ such that $L(\varphi)=f\ast \varphi\in E'$, for every $\varphi\in\mathcal{D}(\mathbb{R}^{n})$. Theorem \ref{karak} yields $f\in\mathcal{D}'_{E'_{\ast}}$.
\end{proof}

One can readily adapt the proof of Theorem \ref{karak} to show the following characterizations of bounded subsets and convergent sequences of $\mathcal{D}_{E'_{\ast}}'$.

\begin{corollary}
\label{cor:bounded} The following properties are equivalent:
\begin{itemize}
\item [$(i)$] $B'$ is a bounded subset of $\mathcal{D}_{E'_{\ast}}'$.
\item [$(ii)$] $\psi \ast B'=\{\psi \ast f:\: f\in B'\}$ is a bounded subset of $E'$ $($equiv. of $E'_{\ast}$$)$ for each $\psi\in \mathcal{D}(\mathbb{R}^n)$.
\item [$(iii)$] There are $M>0$ and $N\in\mathbb{N}$ such that every $f\in B'$ admits a representation (\ref{eq:representation}) with continuous functions $f_{\alpha}\in E'_{\ast}\cap UC_{\omega}$ satisfying the uniform bounds $||f_{\alpha}||_{E'}<M$ and $||f_{\alpha}||_{\infty,\omega}<M$ (if $E$ is reflexive, one may choose $f_{\alpha}\in E'\cap C_{\omega}$).
\end{itemize}

\end{corollary}

\begin{corollary}
\label{cor:sequence} Let $\left\{f_{j}\right\}_{j=0}^{\infty}\subset\mathcal{D}_{E'_{\ast}}'$ $($or similarly, a filter with a countable or bounded basis$)$. The following three statements are equivalent:
\begin{itemize}
\item [$(i)$]  $\left\{f_{j}\right\}_{j=0}^{\infty}$ is (strongly) convergent in $\mathcal{D}_{E'_{\ast}}'$.
\item [$(ii)$] $\left\{f_{j}\ast \psi\right\}_{j=0}^{\infty}$ is convergent in $E'$ $($equiv. in $E'_{\ast}$$)$  for all $\psi\in\mathcal{D}(\mathbb{R}^{n})$.
\item [$(iii)$] There are $N\in\mathbb{N}$ and continuous functions $f_{\alpha,j}\in E'_{\ast}\cap UC_{\omega}$ such that $f_j=\sum_{|\alpha|\leq N} f^{(\alpha)}_{\alpha,j}$ and the sequences $\left\{f_{\alpha,j}\right\}_{j=0}^{\infty}$ are convergent in both $E'_{\ast}$ and $L^{\infty}_{\omega}$ (if $E$ is reflexive one may choose $f_{\alpha,j}\in E'\cap C_{\omega}$).
\end{itemize}
Concerning weak$\:^{\ast}$ convergence of sequences, the following three properties are equivalent:
\begin{itemize}
\item [$(i)^{\ast}$]  $\left\{f_{j}\right\}_{j=0}^{\infty}$ is weakly$\:^{\ast}$ convergent in $\mathcal{D}_{E'_{\ast}}'$.
\item [$(ii)^{\ast}$]  $\left\{f_{j}\ast \psi\right\}_{j=0}^{\infty}$ is weakly$\:^{\ast}$ convergent in $E'$ $($equiv. in $(E'_{\ast})_{\sigma(E'_{\ast},E)}$$)$  for all $\psi\in\mathcal{D}(\mathbb{R}^{n})$.

\item [$(iii)^{\ast}$] There are $N\in\mathbb{N}$ and continuous functions $f_{\alpha,j}\in E'_{\ast}\cap UC_{\omega}$ such that $f_j=\sum_{|\alpha|\leq N} f^{(\alpha)}_{\alpha,j}$, the sequences $\left\{f_{\alpha,j}\right\}_{j=0}^{\infty}$ are uniformly convergent over compacts of $\mathbb{R}^{n}$, and the norms $||f_{\alpha,j}||_{E'}$ and $||f_{\alpha,j}||_{\infty,\omega}$ remain uniformly bounded (if $E$ is reflexive, one may choose $f_{\alpha,j}\in E'\cap C_{\omega}$).
\end{itemize}
\end{corollary}

\begin{proof}
We only show that $(ii)$ in Corollary \ref{cor:sequence} implies $(iii)$ (resp., $(ii)^{\ast}$ implies $(iii)^{\ast}$), the rest is left to the reader. Let $K\subset \mathbb{R}^{n}$ be a compact symmetric neighborhood of the origin. To each $f_{j}$, we associate the continuous linear mapping $L_{j}:\mathcal{D}(\mathbb{R}^{n})\to E'$ given by $L_{j}(\psi)=\psi\ast f_{j}$, $\psi\in\mathcal{D}(\mathbb{R}^{n})$. The condition $(ii)$ and the Banach-Steinhaus theorem imply that the sequence $\left\{L_{j}\right\}_{j=0}^{\infty}$ is convergent in the space of continuous linear mappings $L_{b}(\mathcal{D}(\mathbb{R}^{n}),E')$ (resp., in $L_{b}(\mathcal{D}(\mathbb{R}^{n}),E'_{\sigma(E',E)})$), provided with the strong topology. Thus, there exists $N$ such that $\left\{L_{j}\right\}_{j=0}^{\infty}$ converges in $L_{b}(\mathcal{D}^{N}_{K},E')$ (resp., in $L_{b}(\mathcal{D}^{N}_{K},E'_{\sigma(E',E)})$); in particular, $\left\{\psi\ast f_{j}\right\}_{j=0}^{\infty}$ converges in $E'$ (resp., weakly* $E'$) for each fixed $\psi\in \mathcal{D}^{N}_{K}$. If we take $l$ as in the proof of Theorem \ref{karak}, the representation (\ref{eq:6}) gives $f_{j}=\sum_{|\beta|\leq 2l}g^{(\beta)}_{\beta,j}$ with each $\left\{g_{\beta,j}\right\}_{j=0}^{\infty}$ convergent in $E'$ (resp., weakly* convergent in $E'$) because
$\left\{\varsigma_{l}*f_{j}\right\}_{j=0}^{\infty}$ and $\left\{(\chi F_{l})\ast f_{j}\right\}_{j=0}^{\infty}$ are then convergent in $E'$ (resp., weakly* convergent in $E'$). Likewise, another application of the parametrix method, as in the proof of Theorem \ref{karak}, allows us to replace the sequences $\left\{g_{\beta,j}\right\}_{j=0}^{\infty}$ by sequences $\left\{f_{\alpha,j}\right\}_{j=0}^{\infty}$ having the claimed properties.
\end{proof}

Observe that Corollaries \ref{cor:bounded} and \ref{cor:sequence} are still valid if  $\mathcal{D}(\mathbb{R}^{n})$ is replaced by $\mathcal{S}(\mathbb{R}^{n})$.

When $E$ is reflexive, the space $\mathcal{D}_{E}$ is also reflexive. Furthermore, we have:

\begin{proposition}\label{prop:reflexive} If $E$ is reflexive, then $\mathcal{D}_{E}$ is an FS$^{\ast}$-space and $\mathcal{D}'_{E'}$ is a DFS$^{\ast}$-space. In addition, $\mathcal{S}(\mathbb{R}^{n})$ is dense in $\mathcal{D}'_{E'}$.
\end{proposition}
\begin{proof} Let $\mathcal{D}_{E}^{N}$ be the Banach space of distributions such that $\varphi^{(\alpha)}\in E$ for $|\alpha|\leq N$ provided with the norm $||\: \:||_{E,N}$ (cf. (\ref{eq4})). We then have the projective sequence
\begin{equation}\label{projective}
E\leftarrow\mathcal{D}_{E}^{1}\leftarrow\cdots\leftarrow\mathcal{D}_E^{N}\leftarrow\mathcal{D}_E^{N+1}\leftarrow\dots\leftarrow\mathcal{D}_E,
\end{equation}
where clearly $\mathcal{D}_E=\mathrm{proj} \lim_{N} \mathcal{D}_E^{N}$. Using the Hahn-Banach theorem, one readily sees that every $f\in(\mathcal{D}^{N}_{E})'$ is of the form $f=\sum_{|\alpha|\leq N}f^{(\alpha)}_{\alpha}$, with $f_{\alpha}\in E'$. Thus, each Banach space $\mathcal{D}_{E}^{N}$ is reflexive, or equivalently its closed unit ball is weakly compact. The latter implies that every injection in the projective sequence (\ref{projective}) is weakly compact. This implies all the assertions.
\end{proof}

It should be noticed that the convolution of $f\in\mathcal{D}'_{E'_{\ast}}$ and $u\in L^{1}_{\check\omega}$, defined  as  $\left\langle u\ast f, \varphi \right\rangle:=\left\langle f, \check{u}\ast \varphi\right\rangle$,  $\varphi\in \mathcal{D}_{E}$, gives rise to a continuous bilinear mapping $\ast: L^{1}_{\check{\omega}}\times \mathcal{D}'_{E'_{\ast}}\to \mathcal{D}'_{E'_{\ast}}$, as follows from  (\ref{eq5}).
We will show in Section \ref{subsect convolution} that the convolution of elements of $\mathcal{D}_{E'_{\ast}}'$ can be defined with distributions in a larger class than $L^{1}_{\check{\omega}}$, namely, with elements of the space $\mathcal{D}'_{L^{1}_{\check{\omega}}}$ to be introduced in Section \ref{examples}.

%
\section{Examples: The $\eta$-weighted spaces $L^{p}_{\eta}$, $\mathcal{D}_{L^{p}_{\eta}}$, and $\mathcal{D}'_{L^{p}_{\eta}}$}
\label{examples}
In this subsection we discuss some important examples of the spaces $\mathcal{D}_{E}$ and $\mathcal{D}_{E'_{\ast}}'$. They generalize the familiar Schwartz spaces $\mathcal{D}_{L^{p}}$ and $\mathcal{D}'_{L^{p}}$. These particular instances are useful for studying properties of the general $\mathcal{D}_{E'_{\ast}}'$(cf. Subsection \ref{subsect convolution}).

Let $\eta$ be a polynomially bounded weight, that is, a measurable function $\eta:\mathbb{R}^n\rightarrow (0,\infty)$ that fulfills the requirement $\eta(x+h)\leq M\eta(x)(1+|h|)^\tau$, for some $M,\tau>0$. We consider the norms
$$
||g||_{p,\eta}=\left(\int_{\mathbb{R}^{n}}|g(x)\eta(x)|^{p}dx\right)^{\frac{1}{p}} \mbox{for }p\in [1,\infty) \ \ \mbox{ and }\ \ ||g||_{\infty,\eta}=\operatorname*{ess}\sup_{x\in\mathbb{R}^{n}} \frac{|g(x)|}{\eta(x)}\ .
$$
Then the space $L^{p}_{\eta}$ consists of those measurable functions such that $||g||_{p,\eta}<\infty$ (for $\eta=1$, we write as usual $L^{p}$ and $\|\:\:\|_{p}$). The number $q$ always stands for $p^{-1}+q^{-1}=1$ ($p\in[1,\infty]$). Of course $(L^{p}_{\eta})'=L^{q}_{\eta^{-1}}$ if $1< p<\infty$ and $(L^{1}_{\eta})'=L^{\infty}_{\eta}$. The spaces $E=L^{p}_{\eta}$ are clearly translation-invariant spaces of tempered distributions for $p\in[1,\infty)$. The case $p=\infty$ is an exception, because $\mathcal{D}(\mathbb{R}^{n})$ fails to be dense in $L^{\infty}_{\eta}$. In view of Proposition \ref{thgoodspace}, the space $E'_{\ast}$ corresponding to $E=L^{p}_{\eta^{-1}}$ is $E'_{\ast}=E'=L^{q}_{\eta}$ whenever $1<p<\infty$. On the other hand, Proposition \ref{prop3.11} gives that $E'_{\ast}=UC_{\eta}$ for  $E=L^{1}_{\eta}$, where $UC_{\eta}$ is defined as in (\ref{UC}) with $\omega$ replaced by $\eta$.

We can easily determine the Beurling algebra of $L^{p}_{\eta}$.

\begin{proposition} \label{weight prop}Let
$$
\omega_{\eta}(h):= \operatorname*{ess}\sup_{x\in\mathbb{R}^{n}} \frac{\eta(x+h)}{\eta(x)}.
$$
Then
$$||T_{-h}||_{L(L^{p}_{\eta})}=
\begin{cases}
\omega_{\eta}(h) & \mbox{ if } p\in[1,\infty),
\\ \omega_{\eta}(-h)
  & \mbox{ if } p=\infty.
\end{cases}
$$
Consequently, the Beurling algebra associated to $L_{\eta}^{p}$ is $L_{\omega_{\eta}}^{1}$ if $p=[1,\infty)$ and $L^{1}_{\check{\omega}_{\eta}}$ if $p=\infty$.
\end{proposition}
\begin{proof} Assume first that $1\leq p<\infty$.
Clearly, $\omega_{\eta}(h)\geq ||T_{-h}||_{L(L^{p}_{\eta})}.$ Let $\varepsilon>0$ and set
$$A=\left\{x\in\mathbb{R}^{n}:\: \omega_{\eta}(h)-\varepsilon\leq \eta(x+h)/\eta(x) \right\}.$$
The Lebesgue measure of $A$ is positive. Find a compact subset $K\subset A$ with positive Lebesgue measure and let $g$ be the characteristic function of $K$. Then
$$\int_{\mathbb{R}^{n}}|g(x)|^{p}\eta^{p} (x+h)dx\geq (\omega_{\eta}(h)-\varepsilon)^{p}||g||^{p}_{p,\eta},$$
which yields $||T_{-h}||_{L(L^{p}_{\eta})}(h)\geq (\omega_{\eta}(h)-\varepsilon).$ Since $\varepsilon$ is arbitrary, we obtain $\omega_{\eta}(h)=||T_{-h}||_{L(L^{p}_{\eta})}$. The case $p=\infty$ follows by duality.
\end{proof}

We remark that when the logarithm of $\eta$ is a positive measurable subadditive function and $\eta(0)=1$, one easily obtains from Proposition \ref{weight prop} that $\omega_{\eta}=\eta$.

Consider now the spaces $\mathcal{D}_{L^{p}_{\eta}}$ for $p\in[1,\infty]$, defined as in Subsection \ref{test space} by taking $E=L^{p}_{\eta}$. Once again, the case $p=\infty$ is an exception because $\mathcal{D}(\mathbb{R}^{n})$ is not dense  $\mathcal{D}_{L^{\infty}_{\eta}}$. In analogy to Schwartz' notation \cite{S}, we write $\mathcal{B}_{\eta}:=\mathcal{D}_{L^{\infty}_{\eta}}$. Set further $\dot{\mathcal{B}}_{\eta}$ for the closure of $\mathcal{D}(\mathbb{R}^{n})$ in $\mathcal{B}_{\eta}$. We immediately see that $\dot{\mathcal{B}}_{\eta}=\mathcal{D}_{C_{\eta}}$, where $C_{\eta}=\left\{g\in C(\mathbb{R}^{n}): \: \lim_{|x|\to \infty} g(x)/\eta(x)=0\right\}\subset L^{\infty}_{\eta}.$ Observe that the space $E'_{\ast}$ for $E=C_{\eta}$ is $E'_{\ast}=L^{1}_{\eta}$. By Proposition \ref{smooth prop}, $\mathcal{D}_{L^{p}_{\eta}}\subset \dot{\mathcal{B}}_{\check{\omega}_{\eta}}$ for $p\in (1,\infty)$; using the parametrix formula (\ref{eq:6}), one also deduces that $\mathcal{D}_{L^{1}_{\eta}}\subset \dot{\mathcal{B}}_{\check{\omega}_{\eta}}$. Actually, the estimate (\ref{bounds}) gives $\mathcal{D}_{L^{p}_{\eta}}\hookrightarrow \dot{\mathcal{B}}_{\check{\omega}_{\eta}}$ for every $p\in[1,\infty)$.
It follows from Proposition \ref{prop:reflexive} that  $\mathcal{D}_{ L^p_\eta}$ is an FS$^{\ast}$-space and hence reflexive when $p\in(1,\infty)$.

In accordance to Subsection \ref{subsection DE}, the weighted spaces $\mathcal{D}_{L_{\eta}^p}'$ are defined as $\mathcal{D}_{L_{\eta}^p}'= (\mathcal{D}_{L_{\eta^{-1}}^q})'$ where $p^{-1}+q^{-1}=1$ if $p\in (1,\infty)$; if $p=1$ or $p=\infty$, we have $\mathcal{D}_{L_{\eta}^1}'=(\mathcal{D}_{C_{\eta}})'=(\dot{\mathcal{B}}_\eta)'$ and $\mathcal{D}'_{UC_{\eta}}=(\mathcal{D}_{L^{1}_{\eta}})'$. We write $\mathcal{B}'_{\eta}:=\mathcal{D}'_{L^{\infty}_{\eta}}:=\mathcal{D}'_{UC_{\eta}}$ and $\dot{\mathcal{B}}'_{\eta}$ for the closure of $\mathcal{D}(\mathbb{R}^{n})$ in $\mathcal{B}'_{\eta}$. We call $\mathcal{B}'_{\eta}$ the space of $\eta$-bounded distributions. Observe that the $\mathcal{D}'_{L^{p}_{\eta}}$ are DFS$^{\ast}$ spaces and $(\mathcal{D}'_{L^{p}_{\eta}})'=
\mathcal{D}_{L^{q}_{\eta^{-1}}}$ when $1<p<\infty$. Theorem \ref{karak} gives that $\mathcal{S}(\mathbb{R}^{n})$ is dense in $\mathcal{D}'_{L^{1}_{\eta}}$ and Corollaries \ref{cor:bounded} and \ref{cor:sequence} imply that $(\mathcal{D}'_{L^{1}_{\eta}})'=\mathcal{B}_{\eta}$. Using the parametrix method, one deduces as in the proof of Theorem \ref{karak}, that every element of $\dot{\mathcal{B}}'_{\eta}$ is the sum of partial derivatives of elements of $C_{\eta}$ and that $f\in \dot{\mathcal{B}}'_{\eta}$ if and only if $f\ast \psi\in C_{\eta}$; likewise  analogs to Corollaries \ref{cor:bounded} and \ref{cor:sequence} hold for  $\dot{\mathcal{B}}'_{\eta}$. The later implies that $(\dot{\mathcal{B}}'_{\eta})'=\mathcal{D}_{L^{1}_{\eta}}$. Employing Theorem \ref{karak}, Corollary \ref{cor:bounded} and Corollary \ref{cor:sequence},  one sees that $\mathcal{D}'_{L^{p}_{\eta}}\subset \dot{\mathcal{B}}'_{\check{\omega}_{\eta}}$, $1\leq p<\infty$, and that the inclusion is sequentially continuous. Summarizing, we have the embeddings
$\mathcal{D}_{L^{1}_{\omega_{\eta}}}\hookrightarrow\mathcal{D}_{L^{p}_{\eta}}\hookrightarrow \dot{\mathcal{B}}_{\check{\omega}_{\eta}}$ and $\mathcal{D}'_{L^{1}_{\omega_{\eta}}}\hookrightarrow\mathcal{D}'_{L^{p}_{\eta}}\hookrightarrow \dot{\mathcal{B}}'_{\check{\omega}_{\eta}}$ for $1\leq p <\infty,$ and $\dot{\mathcal{B}}_{{\eta}}\hookrightarrow \dot{\mathcal{B}}_{\omega_{\eta}}$ and $\dot{\mathcal{B}}'_{{\eta}}\hookrightarrow \dot{\mathcal{B}}'_{\omega_{\eta}}.$

The multiplicative product mappings $\cdot:\mathcal{D}'_{L^{p}_{\eta}}\times
\mathcal{B}_{\eta}\to \mathcal{D}'_{L^{p}}$ and $\cdot:\mathcal{B}'_{\eta}\times
\mathcal{D}_{L^{p}_{\eta}}\to \mathcal{D}'_{L^{p}}$ are well-defined and hypocontinuous for $1\leq p<\infty$. In particular, $f\varphi$ is an integrable distribution in Schwartz' sense \cite{S} whenever $f\in\mathcal{B}'_{\eta}$ and $\varphi\in\mathcal{D}_{L^{1}_{\eta}}$ or $f\in\mathcal{D}'_{L^{1}_{\eta}}$ and $\varphi\in\mathcal{B}_{\eta}$. If $(1/r)=(1/p_{1})+(1/p_{2})$ with $1\leq r, p_1,p_2<\infty$, it is also clear that the multiplicative product $\cdot:\mathcal{D}'_{L^{p_1}_{\eta_{1}}}\times
\mathcal{D}_{L_{\eta_{2}}^{p_2}}\to \mathcal{D}'_{L^{r}_{\eta_{1}\eta_{2}}}$ is
hypocontinuous. Clearly, the convolution product can always be canonically defined as a bilinear mapping in the following situations, $\ast:\mathcal{D}'_{L^{p}_{\eta}}\times
\mathcal{D}'_{L_{\omega}^{1}}\to \mathcal{D}'_{L^{p}_{\eta}}$, $1\leq p< \infty$, $\ast:\mathcal{B}'_{\eta}\times
\mathcal{D}'_{L_{\check{\omega}_{\eta}}^{1}}\to \mathcal{B}'_{\eta}$, and $\ast:\dot{\mathcal{B}}'_{\eta}\times
\mathcal{D}'_{L_{\check{\omega}_{\eta}}^{1}}\to \dot{\mathcal{B}}'_{\eta}$. Furthermore, it follows from Remark \ref{wrk2} below and \cite[Th\'{e}or. XXVI, p. 203]{S} that these three convolution mappings are continuous. (The continuity also follows from their hypocontinuity and \cite[p. 160]{kotheII} or \cite[Prop. 1.4.3, p. 19]{O}).

\section{Relation between $\mathcal{D}'_{E'_{\ast}}$, $\mathcal{B}'_{\omega}$, and $\mathcal{D}'_{L^{1}_{\check{\omega}}}$ -- Convolution and multiplication}
\label{subsect convolution}

Many of the properties of the spaces $\mathcal{D}_{L^{p}_{\eta}}$ and $\mathcal{D}'_{L^{p}_{\eta}}$ extend to $\mathcal{D}_{E}$ and $\mathcal{D}'_{E'_{\ast}}$ for the general translation-invariant Banach space of tempered distributions $E$ with Beurling algebra $L^{1}_{\omega}$. The next theorem summarizes some of our previous results.
\begin{theorem}\label{th4.13} We have
$\mathcal{D}_{L^{1}_{\omega}}\hookrightarrow\mathcal{D}_{E}\hookrightarrow\dot{\mathcal{B}}_{\check{\omega}}$ and hence the continuous inclusions $\mathcal{D}'_{L^{1}_{\check{\omega}}}\rightarrow\mathcal{D}'_{E'_{\ast}}\rightarrow{\mathcal{B}}'_{\omega}$. When $E$ is reflexive $\mathcal{D}'_{L^{1}_{\check{\omega}}}\hookrightarrow\mathcal{D}'_{E'}\hookrightarrow\dot{\mathcal{B}}'_{\omega}$.
\end{theorem}
\begin{proof}
Notice that Proposition \ref{smooth prop} gives the inclusions $\mathcal{D}_{E}\subseteq \dot{\mathcal{B}}_{\check{\omega}}$. We actually have $\mathcal{D}_{E}\hookrightarrow\dot{\mathcal{B}}_{\check{\omega}}$ because of (\ref{bounds}). The dense embedding $\mathcal{S}(\mathbb{R}^{n})\hookrightarrow \mathcal{D}_{L^{1}_{\omega}}$ and the inequality (\ref{D omega inequality}) from Remark \ref{remark inequality} show that $\mathcal{D}_{L^{1}_{\omega}}\subseteq \mathcal{D}_{E}$ and that (\ref{D omega inequality}) remains true for all $\varphi\in\mathcal{D}_{L^{1}_{\omega}}$. Consequently, $\mathcal{D}_{L^{1}_{\omega}}\hookrightarrow\mathcal{D}_{E}$. By transposition of the latter two dense inclusion mappings, $\mathcal{D}'_{L^{1}_{\check{\omega}}}\rightarrow\mathcal{D}'_{E'_{\ast}}\rightarrow{\mathcal{B}}'_{\omega}$. In the reflexive case, Theorem \ref{karak} gives $\mathcal{D}'_{E'}\subseteq\dot{\mathcal{B}}'_{\omega}$ and therefore $\mathcal{D}'_{L^{1}_{\check{\omega}}}\hookrightarrow\mathcal{D}'_{E'}\hookrightarrow\dot{\mathcal{B}}'_{\omega}$.

\end{proof}

We can now define multiplication and convolution operations on $\mathcal{D}_{E'_{\ast}}'$.
\begin{proposition}\label{conv-prod} The multiplicative products
$
\cdot:\mathcal{D}'_{E'_{\ast}}\times\mathcal{D}_{L^{1}_{\omega}}\to \mathcal{D}'_{L^{1}}
$
and
$
\cdot:\mathcal{D}'_{L^{1}_{\check{\omega}}}\times\mathcal{D}_{E}\to \mathcal{D}'_{L^{1}}
$ are hypocontinuous.
The convolution products are continuous in the following two cases:
$
\ast:\mathcal{D}'_{E'_{\ast}}\times\mathcal{D}'_{L^{1}_{\check{\omega}}}\to \mathcal{D}'_{E'_{\ast}}
$
and
$
\ast:\mathcal{D}'_{E'_{\ast}}\times\mathcal{O}'_{C}(\mathbb{R}^{n})\to \mathcal{D}'_{E'_{\ast}}
$. The convolution
$
\ast:\mathcal{D}'_{E'_{\ast}}\times\mathcal{D}_{\check{E}}\to \mathcal{B}_{\omega}
$ is hypocontinuous; when the space $E$ is reflexive, we have $
\ast:\mathcal{D}'_{E'}\times\mathcal{D}_{\check{E}}\to \dot{\mathcal{B}}_{\omega}
$.
\end{proposition}
\begin{proof} That these bilinear mappings have the range in the stated spaces follows from Theorem \ref{karak} and Theorem \ref{th4.13}. The hypocontinuity of the multiplicative products is a consequence of Theorem \ref{th4.13}. In fact, the bilinear mapping $
\cdot:\mathcal{D}'_{E'_{\ast}}\times\mathcal{D}_{L^{1}_{\omega}}\to \mathcal{D}'_{L^{1}}
$ is hypocontinuous as the composition of the continuous inclusion mapping $\mathcal{D}'_{E'_{\ast}}\times\mathcal{D}_{L^{1}_{\omega}}\to \mathcal{B}'_{\omega}\times\mathcal{D}_{L^{1}_{\omega}}$ and the hypocontinuous mapping $
\cdot:\mathcal{B}'_{\omega}\times\mathcal{D}_{L^{1}_{\omega}}\to \mathcal{D}'_{L^{1}}
$. Likewise, $\cdot:\mathcal{D}'_{L^{1}_{\check{\omega}}}\times\mathcal{D}_{E}\to \mathcal{D}'_{L^{1}}$ is hypocontinuous.
 It is clear that $\ast:\mathcal{D}'_{L^{1}_{\omega}}\times\mathcal{D}_{E}\to \mathcal{D}_{E}$ is hypocontinuous, which, together with Corollary \ref{cor:bounded}, yields the hypocontinuity of $\ast:\mathcal{D}'_{E'_{\ast}}\times\mathcal{D}'_{L^{1}_{\check{\omega}}}\to \mathcal{D}'_{E'_{\ast}}$. Since $\mathcal{D}'_{E'_{\ast}}$ and $\mathcal{D}'_{L^{1}_{\check{\omega}}}$ are DF-spaces, it automatically follows that the bilinear mapping $\ast:\mathcal{D}'_{E'_{\ast}}\times\mathcal{D}'_{L^{1}_{\check{\omega}}}\to \mathcal{D}'_{E'_{\ast}}$ is continuous (cf. \cite[p. 160]{kotheII}). The continuity of $
\ast:\mathcal{D}'_{E'_{\ast}}\times\mathcal{O}'_{C}(\mathbb{R}^{n})\to \mathcal{D}'_{E'_{\ast}}
$ is a direct consequence of the embedding $\mathcal{O}'_{C}(\mathbb{R}^{n})\hookrightarrow \mathcal{D}'_{L^{1}_{\omega}}$. Finally, $\mathcal{B}_{\omega}=(\mathcal{D}'_{L^{1}_{\omega}})'$ and $\ast:\mathcal{D}'_{E'_{\ast}}\times\mathcal{D}'_{L^{1}_{\check{\omega}}}\to \mathcal{D}'_{E'_{\ast}}$ and $\ast:\mathcal{D}'_{L^{1}_{\omega}}\times\mathcal{D}_{E}\to \mathcal{D}_{E}$ are hypocontinuous, whence the hypocontinuity of $
\ast:\mathcal{D}'_{E'_{\ast}}\times\mathcal{D}_{\check{E}}\to \mathcal{B}_{\omega}
$ follows.
\end{proof}

It is worth pointing out that, as a consequence of Proposition \ref{conv-prod}, $f\varphi$ is an integrable distribution in Schwartz' sense \cite{S} if $f\in\mathcal{D}'_{E'_{\ast}}$ and $\varphi\in\mathcal{D}_{L^{1}_{\omega}}$ or if $f\in\mathcal{D}'_{L^{1}_{\check{\omega}}}$ and $\varphi\in\mathcal{D}_{E}$. We end this section with four remarks. In Remarks \ref{rk2} and \ref{rkBarrelledness}, two open questions are posed.

\begin{remark} Let  $(X, \{|\: \:|_{j}\}_{j\in\mathbb{N}_{0}})$ and $(Y, \{\|\: \|_{j}\}_{j\in\mathbb{N}_{0}})$ be two graded Fr\'{e}chet spaces, namely, Fr\'{e}chet spaces with fixed increasing systems of seminorms defining the topology. Recall that a continuous linear mapping $A: (X, \{|\: \:|_{j}\}_{j\in\mathbb{N}_{0}})\to (Y, \{\|\: \|_{j}\}_{j\in\mathbb{N}_{0}})$ is called tame if there are $\nu, j_{0}\in\mathbb{N}$ such that for any $ j\geq j_{0}$ there is $M_{j}>0$ such that $\| A(f)\|_{j}\leq M_{j} | f|_{\nu j}$ for all $f\in X$.

If $l$ is chosen as in the proof of Proposition \ref{smooth prop}, the inequalities (\ref{bounds}) and (\ref{D omega inequality}) actually show that
$$
(\mathcal{D}_{L^{1}_{\omega}}, \{\|\:\|_{1,\omega,N}\}_{N\in \mathbb{N}_{0}}) \hookrightarrow (\mathcal{D}_{E}, \{\|\:\|_{E,N}\}_{N\in \mathbb{N}_{0}}) \hookrightarrow (\dot{\mathcal{B}}_{\check{\omega}}, \{\|\:\|_{\infty,\check{\omega},N}\}_{N\in \mathbb{N}_{0}})
$$
are tame dense embeddings between these graded Fr\'{e}chet spaces.
With the notation used in the proof of Proposition \ref{prop:reflexive}, we obtain in particular the ``Sobolev embedding'' type results $\mathcal{D}^{2l}_{L^{1}_{\omega}} \hookrightarrow E$ and $\mathcal{D}^{2l}_{E}\hookrightarrow C_{\check{\omega}}$.
\end{remark}

\begin{remark} \label{wrk2} The spaces $\mathcal{D}_{L^{p}_{\omega}}$ (resp., $\mathcal{B}_{\omega}$ and $\dot{\mathcal{B}}_{\omega}$) are isomorphic to the Schwartz spaces $\mathcal{D}_{L^{p}}$ (resp., $\mathcal{B}$ and $\dot{\mathcal{B}}$). To construct isomorphisms, first note that the weight $\omega_{0}=\omega\ast \psi$, where $\psi\in\mathcal{D}(\mathbb{R}^{n})$ is a non-negative function, satisfies the bounds
 $ M_{1}\omega (x)\leq \omega_{0}(x)\leq M_{2} \omega (x)$, $x\in\mathbb{R}^{n}$. Furthermore, $\omega_{0}\in \mathcal{B}_{\omega}$. These two facts imply that the multiplier mapping $\varphi\to \varphi \omega_{0}$ is a Fr\'{e}chet space isomorphism from $\mathcal{D}_{L^{p}_{\omega}}$ onto $\mathcal{D}_{L^{p}}$, $1\leq p<\infty$. The same mapping provides isomorphisms $\mathcal{B}\to \mathcal{B}_{\omega}$ and $\dot{\mathcal{B}}\to \dot{\mathcal{B}}_{\omega}$.
\end{remark}
\begin{remark} \label{rk2}Schwartz has pointed out \cite[p. 200]{S} that the spaces $\mathcal{D}_{L^{p}}$ are not Montel. Remark \ref{wrk2} then yields that $\mathcal{D}_{L^{p}_{\omega}}$ are not Montel either, $1\leq p<\infty$. The spaces $\mathcal{B}_{\omega}$ and $\dot{\mathcal{B}}_{\omega}$ can never be Montel because they are not reflexive. When $\omega$ is bounded, it is easy to see that $\mathcal{D}_{E}$ is never Montel. In fact, if $\varphi\in\mathcal{D}(\mathbb{R}^{n})$ is such that $\varphi(x)=0$ for $|x|\geq 1/2$ and $\theta\in\mathbb{R}^{n}$ is a unit vector, then $\{T_{-j\theta}\varphi\}_{j=0}^{\infty}$ is a bounded sequence in $\mathcal{D}_{E}$ without any accumulation point, as follows from the continuous inclusion $\mathcal{D}_{E}\to \mathcal{B}$. In general: Can $\mathcal{D}_{E}$ be Montel?
\end{remark}

\begin{remark}
\label{rkBarrelledness} When $E$ is reflexive, the space $\mathcal{D}'_{E'}$ is barreled, as follows from Proposition \ref{prop:reflexive} because a reflexive space is barreled. In the general case: Is the space $\mathcal{D}'_{E'_{\ast}}$ barreled?
\end{remark}

\end{document}